\documentclass[reqno]{amsart}

\usepackage{amsmath, amsthm, amsfonts, eucal, amssymb}
\usepackage[english]{babel}
\usepackage{verbatim}
\usepackage[utf8]{inputenc}
\usepackage{graphicx}
\usepackage{layout}
\usepackage[all]{xy}
\usepackage{hyperref}
\usepackage[font=small,labelsep=none]{caption}
\usepackage{subcaption}
\usepackage{pinlabel}

\newtheorem{lem}{Lemma}
\newtheorem{thm}{Theorem}

\newtheorem{prop}{Proposition}

\theoremstyle{definition}

\theoremstyle{remark}

\begin{document}

\title[Meridional surfaces in hyperbolic knot exteriors]{Infinitely many meridional essential surfaces of bounded genus in hyperbolic knot exteriors}

\author{Jo\~{a}o Miguel Nogueira}
\address{University of Coimbra, CMUC, Department of Mathematics,  Apartado 3008, 3001-454 Coimbra, Portugal\\
nogueira@mat.uc.pt}
\thanks{This work was partially supported by the Centre for Mathematics of the University of Coimbra - UIDB/00324/2020, funded by the Portuguese Government through FCT/MCTES. This work was also partially suported by the UTAustin$|$Portugal program CoLab.}



\maketitle

\begin{abstract}
We show the existence of an infinite collection of hyperbolic knots where each of which has in its exterior meridional essential planar surfaces of arbitrarily large number of boundary components, or, equivalently, that each of these knots has essential tangle decompositions of arbitrarily large number of strings. Moreover, each of these knots has in its exterior meridional essential surfaces of any positive genus and (even) number of boundary components. That is, the compact surfaces that have a meridional essential embedding into a hyperbolic knot exterior have meridional essential embeddings into each of these hyperbolic knots exteriors.
\end{abstract}

\section{Introduction}

In geometric topology it is a common approach to use codimension one objects, and the resulting decompositions, to study manifold topology. On the study of 3-manifolds, and knot exteriors in particular, since the work of Haken and Waldhausen, it is common to study their topology through decompositions along embedded surfaces. A very important class of surfaces used in these decompositions are embedded essential surfaces, which has motivated research on the properties and existence of these embeddings. A particularly interesting phenomena is the existence of knots with the property that their exteriors have properly embedded essential surfaces of arbitrarily high Euler characteristics. The first examples of knots with this property were given by Lyon \cite{Lyon}, where he proves the existence of fibered knot exteriors each of which with closed essential surfaces of arbitrarily high genus. Later, other contributions by Oertel \cite{Oertel}, Gustafson \cite{G-94}, Ozawa and Tsutsumi \cite{OT-03}, Eudave-Mu\~noz and Neumann-Coto \cite{EN-04}, Li \cite{Li} and more recently Lopez-García \cite{L-15}, for instance, also gave examples of knots having closed essential surfaces of arbitrarily high genus in their exteriors. These results contrast with the number of non-isotopic closed essential surfaces that are acylindrical, that is with no essential annulus in its exterior, in a 3-manifold being finite, as proved by Hass in \cite{H-95}.\\

\noindent
In the mentioned examples the Euler characteristics of the collection is unbounded because of the increasing genus. Along these lines, in \cite{N-15} we have shown that such a collection of surfaces can be of arbitrarily high genus and two boundary components on a prime knot exterior.  We have also proved that a collection of compact surfaces embedded in a particular knot exterior can have arbitrarily large Euler characteristics due to the number of boundary components and not the genus. In fact,  in \cite{N-16} it was shown instead that a particular knot exterior can have essential tangle decompositions of any number of strings, that is, they have meridional planar essential surfaces of any even number of boundary components. This is somewhat surprising as in Proposition 2.1 of \cite{MT-08}, Mizuma and Tsutsumi proved that for a given knot the number of strings in essential tangle decompositions without parallel strings is bounded. Moreover, in \cite{N-18} we use the results from \cite{N-15} and ideas from \cite{N-16} to actually show the existence of knot exteriors where each of which has meridional essential surfaces of any (even) number of boundary components and genus. Hence, all compact surfaces that  have a meridional essential embedding into a knot exterior have one into each knot exterior of \cite{N-18}. 
However, these collections of knot exteriors with meridional essential surfaces of arbitrarily high number of boundary components are not of hyperbolic knot exteriors.\\ 

\noindent For a hyperbolic knot exterior in particular, it is a result attributed to Haken that it cannot have infinitely many surfaces of uniformly bounded Euler characteristics. (See also \cite{JO-84} by Jaco and Oertel, or \cite{H-95} by Hass.) Without a bound on the Euler characteristics, among the collections of knots with meridional essential surfaces in their exterior with arbitrarily large Euler characteristics, a result in  \cite{N-18}  includes hyperbolic knots exteriors in some 3-manifold where the unbounded Euler characteristics of surfaces in the collection comes independently from both the number of boundary components and the genus, but the genus is always higher than zero and the base 3-manifold not necessarily $S^3$. In this paper we  consider the question whether a  hyperbolic knot in the 3-sphere can also have a set of meridional  planar essential surfaces of unbounded Euler characteristics properly embedded in its exterior, that is, with arbitrarily high number of boundary components. In other words, we study if a single hyperbolic knot can have essential tangle decompositions of any number ($\geq 2$) of strings as in \cite{N-15}. We also consider the more general question if the arbitrarily large Euler characteristics in the collection of meridional essential surfaces in a single hyperbolic knot exterior can be independently due from both the genus and the number of boundary components. We prove the existence of a collection of hyperbolic knot exteriors where each of which has in its exterior all possible meridional essential surfaces. That is, the compact surfaces that have a meridional essential embedding into a hyperbolic knot exterior have meridional essential embeddings into each of these hyperbolic knots exteriors. The main results of this paper are summarized in the following theorem.

\begin{thm}\label{main}
There are infinitely many hyperbolic knots in the 3-sphere each of which having in its exterior
\begin{itemize} 
\item[(a)] a meridional essential planar surface with $2n$ boundary components for any integer $n\geq 2$;
\item[(b)] a meridional essential surface of any positive genus with $2n$ boundary components for any positive integer $n$.
\end{itemize}
\end{thm}

\noindent 
The knots of Theorem \ref{main} also have in their exteriors closed essential surfaces of any genus greater than or equal to two. This follows from Theorem 2.0.3 of \cite{CGLS}, or the handle addition lemma \cite{CG-87}, applied to the surfaces with two boundary components in Theorem \ref{main}. Note that the  knot exterior of $8_{16}$ \cite{KnotInfo}, which is an alternating and hyperbolic knot, also has meridional  essential surfaces with two boundary components and any positive genus, together with the corresponding swallow-follow closed surfaces. This derives from $8_{16}$'s 2-string essential tangle decomposition and \cite{L-15}, but it does not have meridional essential planar surfaces of arbitrarily high number of boundary components as its essential tangle decomposition is unique \cite{Oz-98}.\\

\noindent 
The statement of Theorem \ref{main}  also complements and contrasts with the result of Oertel \cite{O-02}, extending work attributed by Oertel to Jaco and Sedgwick, saying in particular that a knot exterior without essential genus one surfaces (with or without boundary) has at most finitely many isotopy classes of compact essential surfaces of uniformly bounded genus. In Theorem \ref{main} we show that each hyperbolic knot exterior of the statement has infinitely many essential surfaces of any genus, and hence has infinitely many isotopy classes of surfaces of bounded genus. Therefore, from the above mentioned Oertel's result, these knot exteriors have necessarily some essential genus one surface (with boundary for the knots being hyperbolic). Such essential genus one surface is  already observed for the knots in Theorem \ref{main}(b). The result of this paper also complement the results of Menasco in \cite{M-84} stating that non-split prime alternating knots (which are hyperbolic excluding the alternating torus knots, as proved in the cited paper) have at most finitely many meridional essential meridionally-incompressible surfaces of uniformly bounded genus. It is relevant to note that all surfaces but the 4-punctured sphere as in the statement of Theorem \ref{main} are meridionally-compressible.
In the opposite direction of Theorem \ref{main} we know that: small knot exteriors \cite{CGLS} have no meridional essential surfaces; tunnel number one knot exteriors \cite{Gordon-Reid} and free genus one knot exteriors \cite{MO-98} have no meridional essential planar surfaces. There are also knots with an unique essential tangle decomposition \cite{Oz-98}.\\ 

\noindent The paper is organized as follows: In Section \ref{section:tangle with surfaces} we show the existence of atoroidal 2-string essential tangles with meridional essential surfaces of any positive genus and two boundary components. These tangles are a base  to construct  in Section \ref{section:hyperbolic knots} a collection of hyperbolic knots which we use to prove Theorem \ref{main} in Section \ref{section:surfaces} recurring to branched surface theory. Throughout this paper all manifolds are orientable, all submanifolds are assumed to be in general position and we work in the smooth category. We use $N(X)$ to denote a regular neighborhood of $X$, and $|X|$ to denote the number of components of $X$.

\section{Twice punctured meridional essential surfaces in 2-string tangle exteriors}\label{section:tangle with surfaces}

In this section we will show the existence of a 2-string essential tangle with meridional essential surfaces of any genus and two boundary components.\\
Let $J$ be an hyperbolic knot with a 2-string essential tangle decomposition with tangles $(Q_+; u_+\cup v_+)$ and $(Q_-; u_-\cup v_-)$, defined by the sphere $Z$. From Theorem 1.2 of \cite{L-15}, and its proof, there are meridional essential surfaces of any positive genus and two boundary components in its exterior. The surfaces are as follows: there are two twice-punctured swallow-follow essential tori with meridional boundary, obtained from $Z$, one in each tangle of the decomposition, denoted $F^+$ and $F^-$. Suppose $F^\pm$ is in $Q_\pm$ (where $\pm$ denotes $+$ or $-$, respectively). The boundary components of  $F^\pm$ are on the same string, say $u_\pm$. Each surface $F^\pm$ has accidental peripheral with meridional slope, that is, it has an essential non-boundary parallel simple closed curve co-bounding an annulus with a meridian of $\partial N(J)$ in $E(J)-F^\pm$, which is referred as \textit{meridional accidental annulus}. A meridional accidental annulus $A$ of $F^+$ associated with the string $v_+$ is \textit{not centered}. That is, by numbering the boundary components of $F^+\cup F^-$ from $\partial A\cap N(J)$ in one direction by $\{1, 2\}$ and the other direction as $\{-1, -2\}$, the boundary components with symmetric numbering, \textit{e. g.} 1 and $-1$, don't belong to the same surface $F^\pm$. (See \cite{L-15}.)\\

Continuing as in \cite{L-15}, we take $n$ parallel copies of $F^+$, $F^+_i$ for $i=1, \ldots, n$, and one copy of $F^-$ and glue them through nested tubing operations, connecting their boundary components following the numeration from $\partial A\cap N(J)$ in both directions of $\partial N(J)$: $\{1, \ldots, k\}$ and $\{-1, \ldots, -k\}$, respectively.  
So, the boundary component $i$ is connected to the boundary component $-i$, for $i={1, \ldots, k-1}$ in this order. The resulting surface $F$ has genus $n+1$ and two meridional boundary components, the boundary components $k$ and $-k$, and is essential in the exterior of $J$. (See \cite{L-15} for more details.)

\begin{figure}[htbp]

\labellist
\small \hair 2pt

\scriptsize
\pinlabel $F^+_1$ at 125 38 
\pinlabel $F^+_2$ at 125 21 
\pinlabel $F^-$ at 185 21 
\pinlabel $Q_+$ at 125 90
\pinlabel $Q_-$ at 30 110
\pinlabel $J$ at 30 53

\pinlabel $1$ at 182 86
\pinlabel $2$ at 152 86
\pinlabel $3$ at 98 86
\pinlabel $-1$ at 48 53
\pinlabel $-2$ at 47 86
\pinlabel $-3$ at 62 86

\endlabellist

\centering
\includegraphics[width=0.65\textwidth]{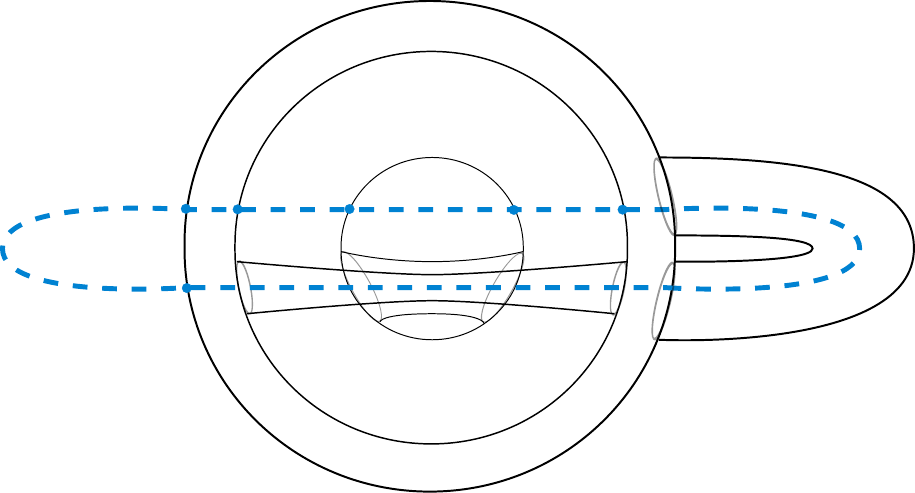}

\caption{: A schematic representation of the surfaces $F^+_1$ and $F^+_2$ in $Q_+$, the surface $F^-$ in $Q_-$, and the numbering of the boundary components of $F^+_1\cup F^+_2\cup F^-$ from a meridional accidental annulus $A$ of $F^+$ in $Q_+$.}
\label{figure:SurfaceF1F2}
\end{figure}
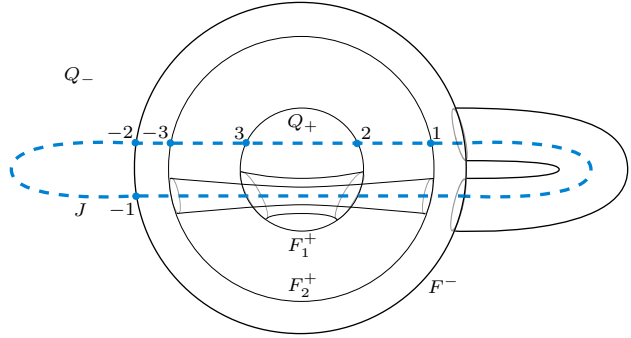

If we order the copies of $F^+$ on this construction from the inside of $Q_+$ to the outside, we have that $k$ belongs to $F^+_1$ and $-k$ to $F^+_2$. (See Figure \ref{figure:SurfaceF1F2}.) Let us consider the surface $S$ bounding a ball $B'$ intersecting $u_+$ at a trivial arc between the points associated with $k$ and $-k$, and the other ball bounded by $S$ let us denote it by $B$, as illustrated schematically in Figure \ref{figure:2tangle}. Denote the string $J-B'\cap u_+$ in $B$ by $s$. Let $Z_i$, for $i=1, 2$, be the copy of $Z$ corresponding to $F_i^+$. Let us consider the space $G$  between $Z_1$ and $Z_2$, and by $G_B$ the intersection $B\cap G$. The space $u_+\cap G_B$ consists of three components. Denote the arc component of $u_+\cap G_B$ connecting $S$ and $Z_1$ by $a_+$ and the arc component of $u_+\cap G_B$ disjoint from $S$ by $a_+'$. Let $t$ be a properly embedded arc in $B$ that lies in the interior of $G_B$. We choose $t$ so that it is trivial in the exterior of $a_+$ in $G_B$ and co-bounds a disk in $G_B$ with an arc in $S$ that intersects $a_+'$ once, being disjoint from $s$ and $F^+_i$, for $i=1,2$, otherwise. Consider the tangle $(B; s\cup t)$. (See Figure \ref{figure:2tangle}.)

\begin{figure}[htbp]

\labellist
\small \hair 2pt

\scriptsize
\pinlabel $F$ at 185 21 
\pinlabel $G_B$ at 150 35 

\pinlabel $S$ at 80 73
\pinlabel $t$ at 145 110
\pinlabel $s$ at 125 78

\pinlabel $B$ at 30 30

\endlabellist

\centering
\includegraphics[width=0.7\textwidth]{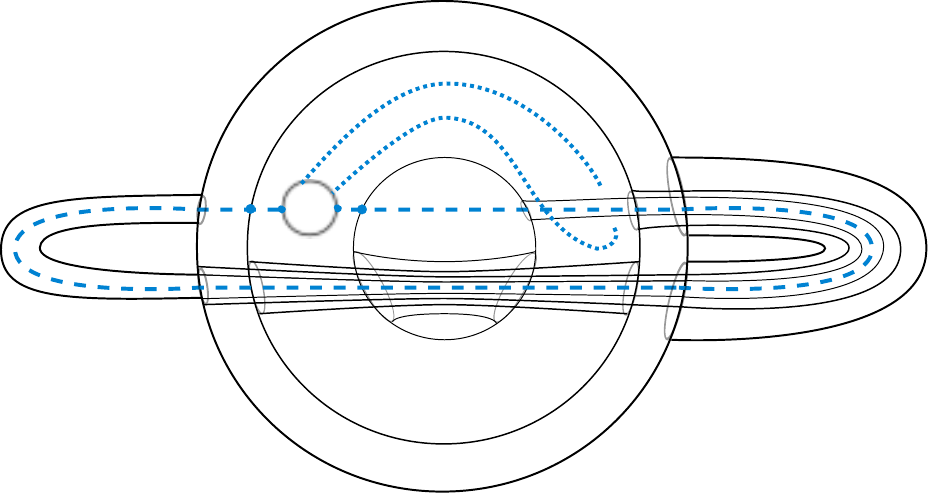}
\caption{: A schematic illustration of the tangle $(B; s\cup t)$ and the surface $F$ in the exterior of of $s\cup t$ in $B$, and the space $G_B$.}
\label{figure:2tangle}
\end{figure}
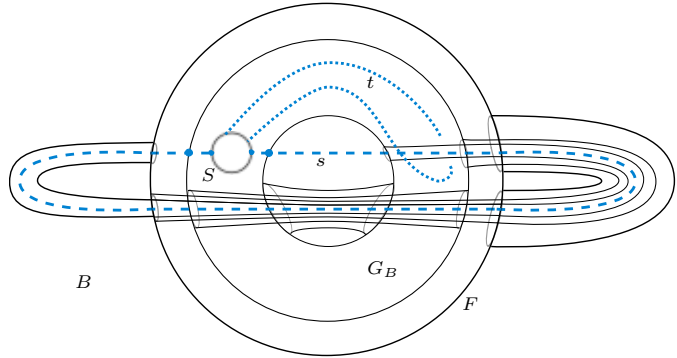

\begin{lem}\label{lemma:surfaces in tangle}
The tangle $(B; s\cup t)$ is essential, atoroidal, and each surface $F$ is essential in the exterior of $s\cup t$ in $B$.
\end{lem}
\begin{proof}
First we observe that each surface $F$ is essential in the exterior of $s\cup t$ in $B$, denoted $E_B$; otherwise, as the arc $B'\cap J$ is a trivial single arc in $B'$, it would be inessential in the exterior of $J$, which is a contradiction.\\
Now we show that the tangle $(B; s\cup t)$ is essential. Suppose, by contradiction, that it is not, and $D$ is a compressing disk for $S-s\cup t$ in $B-s\cup t$. As $Z_1$ and $Z_2$ are essential in $E_B$ we have that $D$ is disjoint from $Z_1$ and $Z_2$: If $D$ intersects $Z_i$ it intersects it in simple closed curves each of which bounds a disk in $D$. By choosing an innermost one in $D$, say $c$, this curve also has to bound a disk $O_c$ in $Z_i$ because $Z_i$ is essential in $E_B$. By choosing an innermost curve of $O_c\cap D$ in $O_c$, we can use it to cut and paste $D$ and eliminate this intersection with $Z_i$. Proceeding in this way we can assume that $D$ is disjoint from $Z_i$, $i=1, 2$. Hence, $D$ is in $G_B$. Let $E$ in $G_B$ be a disk co-bounded by $t$ and an arc in $S$, and intersecting $a_+'$ once at a single point $p$. If $E$ is disjoint from $D$, as $D$ is disjoint from $Z_1\cup Z_2$, it cuts a ball from $B$ containing $E$ which is in $G_B$. But then $D$ separates $E$ from $a_+'$, which is a contradiction as $a_+'$ intersects $E$. Otherwise, assume the intersection of $E$ with $D$ to be non-empty and with minimal number of components $|E\cap D|$. In case there are simple closed curves in $E\cap D$ consider an innermost one in $E$, say $c$. Let $E_c$ be the disk cut by $c$ in $E$.  If $E_c$ is disjoint from $p$, by cutting and pasting $D$ along this disk we reduce  $|E\cap D|$, contradicting its minimality. If $E_c$ contains $p$, together with the disk bounded by $c$ in $D$ we have a sphere in $S^3$ intersecting $J$ once, which contradicts, for instance, the 3-dimensional Schoenflies theorem. Then, $E\cap D$ is a collection of arcs. Let $b$ be an outermost arc of $E\cap D$ in $E$ and $E'$ the corresponding outermost disk. Following as for the simple closed curve case we get a contradiction. Hence, the tangle  $(B; s\cup t)$ is essential.\\
At last we prove that $(B; s\cup t)$ is atoroidal. Let $T$ be an essential torus in $B-s\cup t$. As $J$ is hyperbolic the torus $T$ is inessential in $B-s$. Then $T$ is disjoint from $Z_1$ and $Z_2$, as otherwise either some annulus in $T-T\cap Q_\pm$ is inessential in $Q_\pm-u_\pm\cup v_\pm$ and we can reduce $|T\cap Q_\pm|$ or a component of $T\cap \partial Q_\pm$ bounds a disk in $\partial Q_\pm - u_\pm\cup v_\pm$, contradicting $T$ being essential in $B-s\cup t$.
Hence, we have that the torus $T$ is disjoint from $Z_1$ and $Z_2$, but as the 2-string tangles $(Q_\pm; u_\pm\cup v_\pm)$  are  atoroidal, for $J$ being hyperbolic, the torus $T$ is in $G_B$. 
As $t$ is trivial in the exterior of $a_+$ in $G_B$ there is a disk $O$ co-bounded by $t$ and an arc in $\partial (G_B-N(a_+))$ and disjoint from $s$. If $T$ intersects $O$ by considering an innermost curve $c$ of this intersection in $O$ we obtain a compressing disk for $T$ in $B-s\cup t$ or we eliminate that intersection in case $c$ bounds a disk in $T$ by cutting and pasting.  
If $D_T$ is a compressing disk of $T$ in $B-s$. Suppose $D_T$ intersects $O$ and assume that $|D_T\cap O|$ is minimal. If there is a simple closed curve in $D_T\cap O$, consider an innermost one in $O$ with innermost disk $O'$. By cutting and pasting $D_T$ along $O'$ we eliminate this intersection, reducing $|D_T\cap O|$ and contradicting its minimality. If $D_T\cap O$ has no simple closed curves, let $b$ be an outermost arc of the intersection in $O$ with corresponding outermost disk $O'$. The ends of $b$ are in $t$, as $D_T$ is disjoint from $\partial G_B$, and we can isotope $t$ along $O'$ through $D_T$ eliminating this intersection and contradicting the minimality of $|D_T\cap O|$. Then $D_T$ is disjoint from $O$, and, therefore, contradicting $T$ being essential in $B-s\cup t$. So, $(B; s\cup t)$ is atoroidal.  
\end{proof}

\section{A collection of hyperbolic knots}\label{section:hyperbolic knots}

Let us consider a knot $K$ decomposed into two 2-string tangles by the sphere $S$: a 2-string essential atoroidal tangle $(B; s\cup t)$ with meridional essential surfaces of any positive genus and exactly two boundary components in $\partial N(s)$, which from Lemma  \ref{lemma:surfaces in tangle} we know exist; and a 2-string  tangle $(B'; s'\cup t')$.  Denote these tangles by $\mathcal{T}$ and $\mathcal{T}'$ respectively, as illustrated schematically in Figure \ref{figure:hyperbolic knots}(a). We assume $\mathcal{T}'$ has a decomposition into a 2-string essential atoroidal tangle $(B'_1; s'_1\cup t'_1)$ and a 3-string essential atoroidal tangle $(B'_2; s'_2\cup s'_3\cup t'_2)$, denoted $\mathcal{T}'_1$ and  $\mathcal{T}'_2$ respectively, by a  disk $D$, with $\partial D$   separating in $S$ the end of $s$ that meets $t'$ from the other ends of $s\cup t$, as illustrated schematically in Figure \ref{figure:hyperbolic knots}(b). (See the Appendix for examples of $n$-string atoroidal essential tangles.) Denote the set of equivalence classes, up to ambient isotopy, of these knots by $\mathcal{K}$. In Figure \ref{figure:example} we have a diagram of a knot in $\mathcal{K}$ with low crossing number.

\begin{figure}[htbp]

\labellist
\small \hair 2pt
\pinlabel (a) at 10 0
\scriptsize
\pinlabel $B$ at 115 220 
\pinlabel $s$ at 135 190 
\pinlabel $t$ at 135 150 
\pinlabel $B'$ at 30 220 
\pinlabel $s'$ at 100 40 
\pinlabel $t'$ at 200 110 
\pinlabel $S$ at 105 115 

\small
\pinlabel (b) at 415 0
\scriptsize
\pinlabel $B'_2$ at 485 75 
\pinlabel $s'_3$ at 515 40 
\pinlabel $s'_2$ at 490 13 
\pinlabel $B'_1$ at 690 165 
\pinlabel $s'_1$ at 680 90 
\pinlabel $t'_1$ at 680 125 
\pinlabel $S$ at 505 115 
\pinlabel $D$ at 680 35 

\endlabellist

\centering
\includegraphics[width=0.8\textwidth]{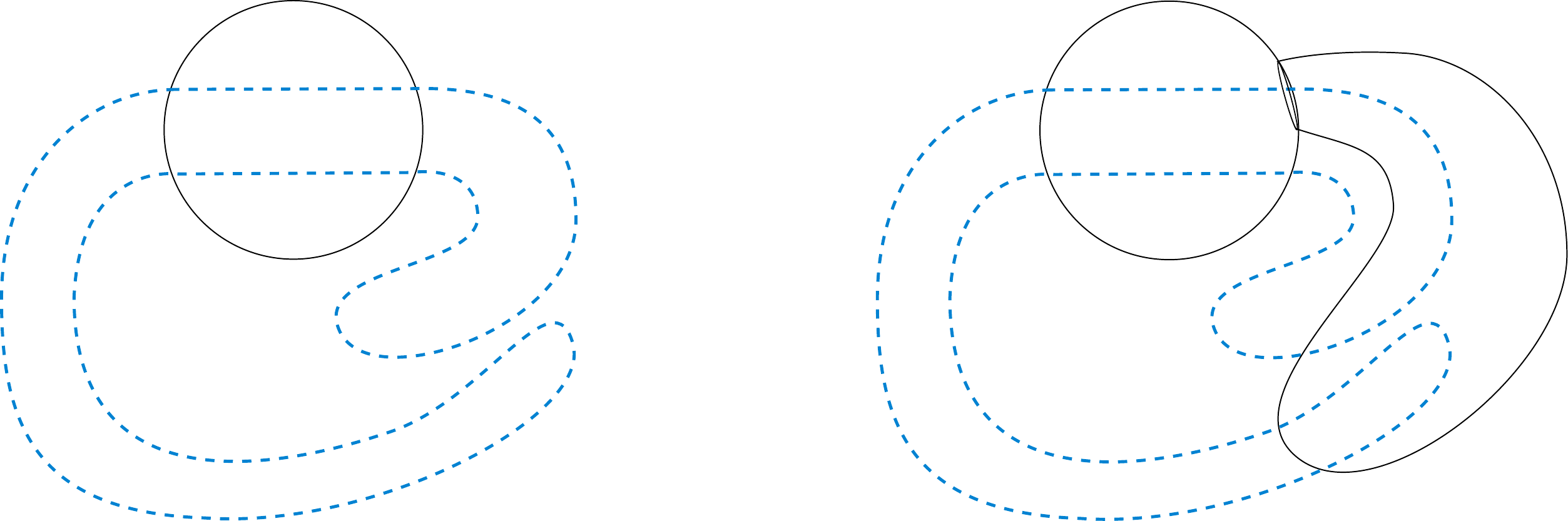}
\caption{: Schematic representation of the tangles $\mathcal{T}$, $\mathcal{T'}$, $\mathcal{T'}_1$ and $\mathcal{T'}_2$.}
\label{figure:hyperbolic knots}
\end{figure}

\begin{figure}[htbp]

\labellist

\scriptsize
\pinlabel $S$ at 155 220 

\endlabellist

\centering
\includegraphics[width=0.4\textwidth]{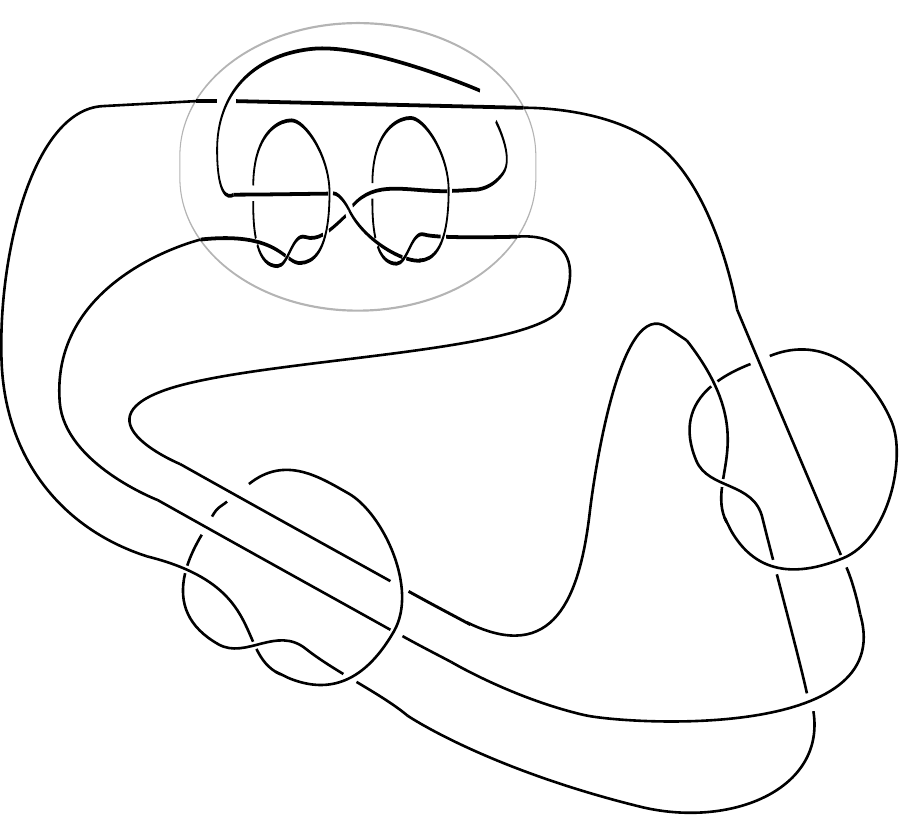}
\caption{: Diagram of a knot of $\mathcal{K}$.}
\label{figure:example}
\end{figure}

We denote the exterior of the strings of $\mathcal{T}$, $\mathcal{T'}$, $\mathcal{T'}_1$ and $\mathcal{T'}_2$  in $B$, $B'$, $B'_1$ and $B'_2$ by  $E_B$, $E_{B'}$, $E_{B'_1}$ and $E_{B'_2}$, respectively. Let $P_D$ denote the punctured disk $D\cap E_{B'}$.

\begin{lem}\label{lemma:punctured disk}
The thrice punctured disk $P_D$ is essential in $E_{B'}$.
\end{lem}
\begin{proof}
This lemma follows from $P_D$ being essential in $E_{B'_i}\cap \partial B'_i$ and from $E_{B'_i}\cap \partial B'_i$ being essential in $B'_i$ (for $\mathcal{T'}_i$ being an essential tangle), for $i=1,2$.
\end{proof}

\begin{lem}\label{lemma: T' is atoroidal}
The tangle $\mathcal{T}'$ is essential and atoroidal.
\end{lem}
\begin{proof}
Suppose, by contradiction, that $\mathcal{T}'$ is inessential. Let $E$ be a disk in $B'$ separating the strings $s'\cup t'$. Let us consider the intersection of $E$ with $P_D$ and assume that $|E\cap P_D|$ is minimal among all choices of $E$. If $E\cap P_D$ has circle components, let $\alpha$ be an innermost  one  in $E$, otherwise let $\alpha$ be an outermost arc of $E\cap P_D$ in $E$, and denote by $O_E$ the corresponding innermost/outermost disk. Hence, as $P_D$ is essential in $E_{B'}$, from Lemma \ref{lemma:punctured disk}, the disk $O_E$ cannot be a (boundary) compressing disk for $P_D$. Then, $\alpha$ cuts a disk $O_D$ from $P_D$.  By considering an innermost/outermost curve/arc of $O_D\cap E$ in $O_D$, say $c$, and the corresponding innermost/outermost disk $O_c$, then by cutting and pasting $E$ with respect to $O_c$ at least one of the resulting disks separates the strings of $\mathcal{T}'$ reducing $|E\cap P_D|$ and contradicting its minimality. Hence $\mathcal{T}'$ is essential.\\
Suppose now, also by contradiction, that $B'-s'\cup t'$ has an embedded essential torus $P$. As $\mathcal{T}'_1$ and  $\mathcal{T}'_2$ are atoroidal, $P$ intersects $D$.  Let us consider the intersection of $P$ with $D$ and assume that $|P\cap D|$ is minimal among all choices of $P$. Suppose a component of $P\cap D$ bounds a disk  in $P$ and consider an innermost disk $O_P$ among these in $P$. As $P_D$ is essential in $E_{B'}$, $\partial O_P$ bounds a disk $O_D$ in $P_D$. As in a previous argument, consider an innermost curve of $O_D\cap E$ in $O_D$, say $c$, and the corresponding innermost disk $O_c$. As $P$ is essential in $B'-s'\cup t'$, the boundary of $O_c$ bounds a disk in $P$. Then by cutting and pasting $P$ with respect to $O_c$ we reduce $|P\cap P_D|$, contradicting its minimality.  Hence, all components of $P\cap D$ in $P$ are parallel essential curves in $P$ and the components of $P-P\cap D$ are all annuli. As $P$ is a torus and $D$ is separating, there are at least two annuli and one of these, say $A$, is in $B'_2$ and has both boundary components in $D$. As  each string of $\mathcal{T}'_2$ has only one end in $D$, the annulus $A'$ that $\partial A$ bounds in $D$  is disjoint from $s'\cup t'$. Let $P'$ be the torus $A\cup A'$.  In case the component $P'$ bounds in $B'_2$ is not solid torus, as $A$ is incompressible in $B'_2-s'_2\cup t'_2$, this implies that $P'$ is incompressible in $B'_2-s'_2\cup t'_2$, contradicting $\mathcal{T'}_2$ being atoroidal. Otherwise, the component $P'$ bounds in $B'_2$ is a solid torus, and, therefore, we have that $A$ is boundary compressible, and hence is parallel to $A'$. Then, we can isotope $P$ along $A$ through $A'$ and reduce $|P\cap D|$, contradicting its minimality. That is, $\mathcal{T}'$ is atoroidal. 
\end{proof}

\begin{lem}\label{lemma:hyperbolic}
The knot $K$ is hyperbolic.
\end{lem}
\begin{proof}
The knot $K$ has a 2-string atoroidal essential tangle decomposition, as $\mathcal{T}$ is essential and atoroidal, by definition, and $\mathcal{T}'$ is essential and atoroidal, from Lemma \ref{lemma: T' is atoroidal}. The statement then follows from Theorem 1 in \cite{S-83} by Soma.
\end{proof}

\begin{lem}\label{lemma:surfaces}
The exterior of the knot $K$ has meridional essential surfaces  in $B$ of any positive genus and exactly two boundary components in $\partial N(s)$.
\end{lem}
\begin{proof}
Let $F$ be a surface in $B$ as in Lemma \ref{lemma:surfaces in tangle}, which can be chosen of any positive genus and two boundary components in $\partial N(s)$.\\
As $F$ is essential in $E_B$ and $\partial B$ is essential in $E(K)$, it follows that $F$ is essential in $E(K)$.
\end{proof}

\begin{lem}\label{lemma:infinite}
The set $\mathcal{K}$ is an infinite collection of hyperbolic knots.
\end{lem}
\begin{proof}
From Lemma \ref{lemma:hyperbolic}, each element of $\mathcal{K}$ is an hyperbolic knot.\\
Consider a 2-string tangle $\mathcal{T}$  decomposed by a collection  of disjoint disks $\{X_1, \ldots, X_n\}$, $n\in \mathbb{N}$, in $B$, each intersecting $s\cup t$ at two points, one point from each string, such that the 2-string tangle $\mathcal{T}_i$ cut by $X_i\cup X_{i+1}$ from $\mathcal{T}$ is essential and atoroidal. For  $\mathcal{T}_i$ being essential, also such a $\mathcal{T}$ is essential. From Theorem 2 of \cite{S-83}, as $\mathcal{T}_i$, $i=1, \ldots, n$, is atoroidal, we have that $\mathcal{T}$ is atoroidal.  If we consider $\mathcal{T}_1$ as in Lemma \ref{lemma:surfaces in tangle}, we have that the surfaces in the exterior  $\mathcal{T}_1$ remain essential in the exterior of $\mathcal{T}$. Hence, the tangle $\mathcal{T}$ has meridional essential surfaces of any positive genus and exactly two boundary components in the boundary of a regular neighborhood of one string. Therefore, there are knots in $\mathcal{K}$ with $(B;s\cup t)$ as $\mathcal{T}$. Denote by $K_n$ a knot of the collection $\mathcal{K}$ with the tangle $\mathcal{T}$ as just defined, for each $n\in \mathbb{N}$. Note that $K_n$ has at least $n$ disjoint non-parallel meridional planar essential surfaces in its exterior. Therefore, from Kneser-Haken finiteness theorem for essential surfaces with boundary properly embedded in 3-manifolds, as for instance in Jaco's book \cite{Jaco}, and from knots being determined by their complements \cite{GL-89}, we have that the set of equivalence classes of knots $\{K_n: n\in \mathbb{N}\}$ contains an infinite subset of non-equivalent knots. As $\mathcal{K}$ contains this set, it also has infinite cardinality.
\end{proof}

\section{Meridional essential surfaces from a branched surface}\label{section:surfaces}

In this section we prove the Theorem \ref{main} by showing that each knot $K$ in $\mathcal{K}$ is as in the statement of the theorem.

\begin{lem}\label{lemma: boundaries}
Suppose that the exterior of $K\in \mathcal{K}$ has a meridional essential surface $F$ in $B$ of genus $g$ and $2n$ boundary components with exactly two boundary components in the boundary of $N(s)$. Then the exterior of $K$ has meridional essential surfaces of genus $g$ and $2b$ boundary components for every $b\geq n$. 
\end{lem}
\begin{proof}
We define the surfaces in the exterior of $K$ as in the statement of the lemma, denoted by $F_0, \ldots F_j, \ldots$, where $F_j$ has $2(n+j)$ boundary components, and afterwards we prove they are essential in $E(K)$. We will use branched surface theory for this process.\\
The surface $F_0$ is defined as $F$ in $E(K)$.
We isotope the closest boundary component of $F$ in $\partial N(s)$ to $\partial D$ (which we can isotope to be in $\partial N(s)$) and identify it with $\partial D$. We define $F_1$ as the resulting surface from the union of $F$ and $P_D$ along $\partial D$, which after an isotopy if needed it is properly embedded in the exterior of $K$ .


\begin{prop}\label{prop:S2-S3}
The surfaces $F_0$ and $F_1$ are meridional essential surfaces in $E(K)$.
\end{prop}
\begin{proof}
The surface $F_0$ is essential in $E(K)$ from Lemma \ref{lemma:surfaces}. 
Suppose, by contradiction, that $F_1$ is inessential in $E(K)$. Considering a (boundary) compressing disk for $F_1$ in $E(K)$, and its intersection with $S$, through an innermost curve$\backslash$outermost arc argument, we obtain a (boundary) compressing disk for $F$ in $E_B$ or for $P_D$ in $E_{B'}$, which is a contradiction to $F$ being essential in $E_B$ or $P_D$ being essential in $E_{B'}$.
\end{proof}

For the remaining surfaces, $F_j$ for $j \geq 2$, we will call on branched surfaces. First, we start by revising the definitions and result relevant to this paper from Oertel's work in \cite{O-84}, and also Floyd and Oertel's work in \cite{Floyd-Oertel}.\\
A \textit{branched surface} $R$ with generic branched locus is a compact space locally modeled on Figure \ref{branchedmodel}(a). Hence, a union of finitely many compact smooth surfaces in a $3$-manifold $M$, glued together to form a compact subspace of $M$ respecting the local model, is a branched surface. We denote by $N=N(R)$ a fibered regular neighborhood of $R$ (embedded) in $M$, locally modelled on Figure \ref{branchedmodel}(b). The boundary of $N$ is the union of three compact surfaces $\partial_h N$, $\partial_v N$ and $\partial M\cap \partial N$, where a fiber of $N$ meets $\partial_h N$ transversely at its endpoints and either is disjoint from $\partial_v N$ or meets $\partial_v N$ in a closed interval in its interior. We say that a surface $S$ is \textit{carried} by $R$ if it can be isotoped into $N$ so that it is transverse to the fibers. Furthermore, $S$ is carried by $R$ with \textit{positive weights} if $S$ intersects every fiber of $N$. If we associate a weight $w_i\geq 0$ to each component on the complement of the branch locus in $R$ we say that we have an $\textit{invariant measure}$ provided that the weights satisfy \textit{branch equations} as in Figure \ref{branchedmodel}(c). Given an invariant measure on $R$ we can define a surface carried by $R$, with respect to the number of intersections between the fibers and the surface. We also note that if all weights are positive then the surface carried can be isotoped to be transverse to all fibers of $N$, and hence is carried with positive weights by $R$.

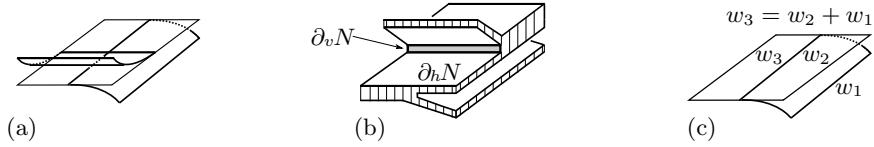
\begin{figure}[htbp]

\labellist
\small \hair 0pt
\pinlabel (a) at 3 -5

\pinlabel (b) at 173 -5

\pinlabel (c) at 335 -5

\pinlabel  $\partial \text{ }N$ at 207 23
\pinlabel \tiny $h$ at 205 21

\pinlabel  $\partial \text{ }N$ at 155 40
\pinlabel \tiny $v$ at 153 37

\pinlabel $w_3=w_2+w_1$ at 383 50
\pinlabel $w_1$ at 407 14
\pinlabel $w_2$ at 391 30
\pinlabel $w_3$ at 368 30

\endlabellist
\centering
\includegraphics[width=0.9\textwidth]{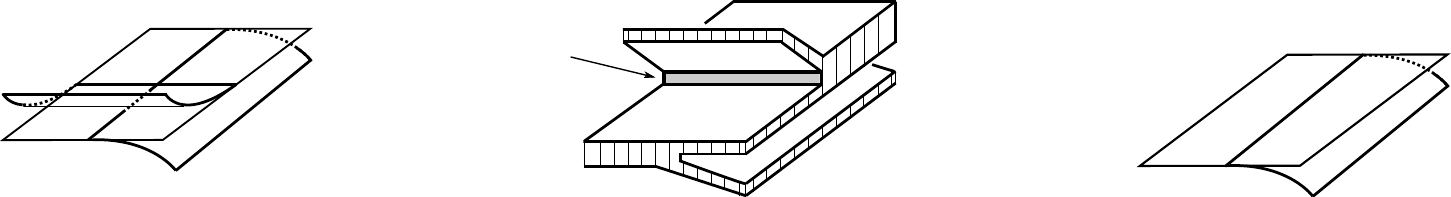}
\caption{: Local model for a branched surface, in (a), and its regular neighborhood, in (b).}
\label{branchedmodel}
\end{figure}

A \textit{disc of contact} is a disc $O$ embedded in $N$ transverse to fibers and with $\partial O\subset \partial_v N$. A \textit{half-disc of contact} is a disc $O$ embedded in $N$ transverse to fibers with $\partial O$ being the union of an arc in $\partial M\cap \partial N$ and an arc in $\partial_v N$. A \textit{monogon} in the closure of $M-N$ is a disc $O$ with $O\cap N=\partial O$ which intersects $\partial_v N$ in a single fiber. (See Figure \ref{monogon}.) We say a branched surface $R$ in M contains a \textit{Reeb component} if $R$ carries a compressible torus or properly embedded annulus, transverse to the fibers of $N$, bounding a solid torus in $M$. (This is a weaker version of the definition of Reeb component in \cite{O-84} by Oertel.)

\begin{figure}[htbp]

\labellist
\small \hair 0pt
\pinlabel (a) at 3 -7

\scriptsize
\pinlabel monogon at 100 30

\small
\pinlabel (b) at 173 -7

\scriptsize
\pinlabel  \text{disk of} at 240 40

\pinlabel \text{contact} at 240 34

\endlabellist
\centering
\includegraphics{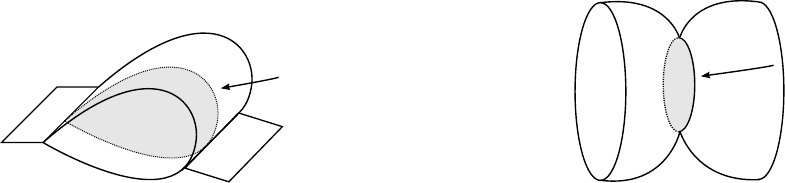}
\caption{: Illustration of a monogon and a disk of contact on a branched surface.}
\label{monogon}
\end{figure}
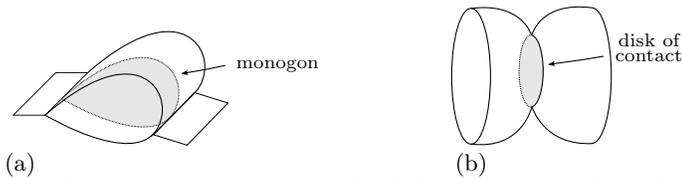

A branched surface embedded $R$ in $M$ is said \textit{incompressible} if it satisfies the following three properties:
\begin{itemize}
\item[(i)] R has no disk of contact or half-disks of contact;
\item[(ii)] $\partial_h N$ is incompressible and  boundary incompressible in the closure of $M-N$, where a boundary compressing disk is assumed to have boundary defined by an arc in $\partial M$ and an arc in $\partial_h N$;
\item[(iii)] There are no monogons in the closure of $M-N$; 
\end{itemize}

and \textit{without Reeb components} if it satisfies the following property: 

\begin{itemize}
\item[(iv)] B doesn't carry a Reeb component.
\end{itemize}

The following theorem proved by Oertel in \cite{O-84} let us infer if a surface carried by a branched surface is essential. Note that condition (iv), $R$ not carrying a torus or an annulus cutting a solid torus from $M$, implies the non-existence of Reeb components in the sense of Oertel \cite{O-84}.

\begin{thm}[Oertel, \cite{O-84}]\label{Oertel}
If $R$ is an incompressible branched surface without Reeb components (i.e. satisfies (i)-(iv)) and $R$ carries some surface with positive weights then any surface carried by $R$ is essential.
\end{thm}

We now proceed to define the surfaces $F_j$, $j\geq 2$, and prove that they are essential, through a branched surface under the conditions of Theorem \ref{Oertel} that we proceed to construct.\\
Let $s'_2$ be the string of $\mathcal{T}'_2$ connecting an end of $s$ to an end of $s'_1$ in $\mathcal{T}'_1$. (See Figure \ref{figure:hyperbolic knots}.) Assume that the regular  neighborhood of each string in each tangle is strictly contained in $N(K)$. Denote by $e_i$, $i=1, 2$, the boundary components of $F$ in $N(s)$ closer  to $s'_i$, $i=1, 2$, respectively.\\
Denote by $A$ the annulus component of $\partial N(K)-F$ in the boundary of the component of $N(K)-S$ which contains $s$. The boundary components of $A$ bound two disjoint annuli in $F$ with $e_i$, $i=1, 2$, respectively, and denote by $b_i$, $i=1, 2$, the corresponding boundary components of $A$. Denote by $A'$ the annulus component  of $\partial N(K)-F\cup D$ in the boundary of the component of $N(K)-F\cup D$ which contains $s'_2$. We assume that $\partial D$ is  $b_1$, and also that $\partial A'\cap F$ is $b_2$, through an isotopy if necessary. Let $A_D$ and $P'_D$ be the annulus and punctured disk component of $P_D-P_D\cap A'$, respectively. Denote also by $F$ the surface obtained from $F$ by removing the two disjoint annuli bounded by $b_i\cup e_i$, for $i=1, 2$, in $F$.\\
Consider the union of $F$, the annuli $A$, $A'$ and $A_D$, and the punctured disk $P'_D$, and denote the resulting space by $R$.
We smooth the space $R$ on the intersections of the surfaces $F$, $A$, $A'$, $A_D$ and $P'_D$ as explained next. The annuli $A'$ and $A_D$ are smoothed on $P_D\cap A'$ towards $P'_D$. The punctured disk $P'_D$ is smoothed along $b_1$ towards $F$. The annulus  $A$ is smoothed along $b_1$ towards $P'_D$, and is smoothed along $b_2$ towards $F$. The annulus  $A'$ is smoothed along $b_2$ towards $e_2$ and $A$. From the construction, the space $R$ is a branched surface with sections denoted naturally by $F$, $A$, $A'$, $A_D$ and $P'_D$, as illustrated in Figure \ref{figure: branched surface}.

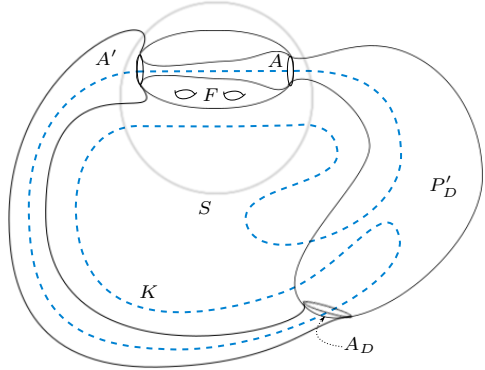
\begin{figure}[htbp]

\labellist
\small \hair 2pt

\scriptsize
\pinlabel $P'_D$ at 310 130
\pinlabel $A'$ at 70 222
\pinlabel $F$ at 144 195
\pinlabel $A$ at 190 219
\pinlabel $A_D$ at 250 17

\pinlabel $S$ at 140 115
\pinlabel $K$ at 100 55

\endlabellist

\centering
\includegraphics[width=0.5\textwidth]{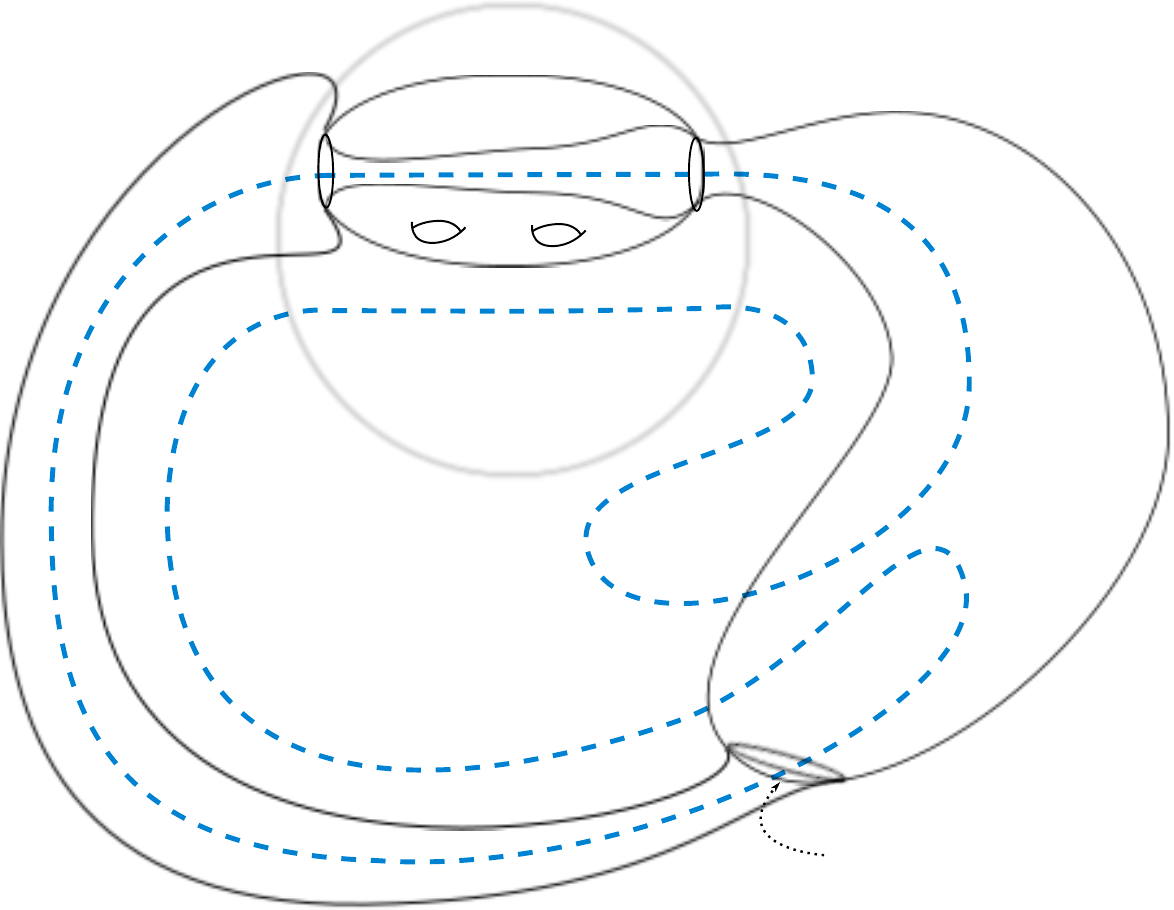}
\caption{: A schematic representation of the branched surface $R$.}
\label{figure: branched surface}
\end{figure}

\begin{prop}\label{proposition:incompressible}
The branched surface $R$ is incompressible and without Reeb components.
\end{prop}
\begin{proof}
First we observe that $R$ doesn't carry a Reeb component. In fact, if $R$ would carry a torus $T$, this torus couldn't be transverse to the regular neighborhood of sections of $R$ with boundary components, or of sections of $R$ of positive genus. In this case, it could be transverse only to regular neighborhoods of sections $A$ and $A'$, which are connected only through the branched locus component corresponding to $b_2$. Hence, $A\cup A'$ cannot carry a closed surface.  If $R$ would carry a properly embedded annulus $O$, this annulus had to be transverse to the regular neighborhood of sections of $R$ with boundary components, that is $F$, $A_D$  or $P'_D$. In case $O$ is transverse to the regular neighborhood of $A_D$ or $F$, as it is properly embedded, it has to be transverse to $P'_D$ as well, but in this case it wouldn't be an annulus, which is a contradiction. If $O$ is transverse to the regular neighborhood of $P'_D$, but not to the one of $A_D$ or $F$, then it is also transverse to the one of $A'$ and $A$, each with weight 1 as $O$ has only two boundary components. But in this case $O$ would be a twice punctured torus, which is a contradiction. Therefore, $R$ doesn't carry an annulus and a torus, and, consequently, it has no Reeb component.																																

Now we prove that $R$ is incompressible in $E(K)$. First observe that there are no (half) disks of contact as no circle on the branched locus of $R$ bounds a disk in $\partial_h N(R)$ and there are no properly embedded arcs on the branched locus of $R$. The space $N(R)$ decomposes $E(K)$ into three components: a component cut from $E(K)$ by $F$ and $A$, denoted $X$; a component cut from $E(K)$ by $A_D$, $P'_D$, $A$ and $A'$, denoted $X_1$; a component cut from $E(K)$ by $F$, $A'$ and $P'_D$, denoted $X_2$.

As $b_1$ corresponds to $\partial_v N(R)$ in $\partial X$, we have that $\partial X\cap \partial_h N(R)$ corresponds to the component cut by $F$ in  $E_B$. Therefore, as $F$ is essential in $E_B$, $\partial X\cap \partial_h N(R)$  is incompressible and boundary incompressible in $X$; and there are no monogons in $X$, otherwise, as $s$ is parallel to $A$, $s$ would be parallel to $F$, contradicting $F$ being essential in $E_B$.

Similarly, $b_2$ corresponds to $\partial_v N(R)$ in $\partial X_2$, and we have that $\partial X_2\cap \partial_h N(R)$ corresponds to $F_1$ in the exterior of $X$ in $E(K)$. From Proposition \ref{prop:S2-S3}, $F_1$ is essential in $E(K)$. Following as in the argument for $X$, we have that $\partial X_2\cap \partial_h N(R)$  is incompressible and boundary incompressible in $X_2$, and there are no monogons in $X_2$.

In the case of $X_1$, $\partial A'\cap A_D$ corresponds to $\partial_v N(R)$ in $\partial X_1$, and $X_1$ corresponds to the exterior of $s'_1\cup t'_1$ in $B'_1$, with $\partial B'_1$ corresponding to two disks singularly glued along $b_1$ corresponding to $\partial_v N(R)$ in $\partial X_1$. Hence, there are no monogons in $X_1$, because any disk intersecting $b_1$ in $\partial X_1$ would have to intersect it at least twice. Also, as $\mathcal{T'}_1$ is essential, $\partial X_1\cap \partial_h N(R)$  is incompressible and boundary incompressible in $X_1$.

Hence, $\partial_h N(R)$ is incompressible and boundary incompressible in $E(K)$, and there are no monogons of $R$ or disks of contact  in $R$. Therefore, $R$ is an incompressible branched surface.
\end{proof}

Let us define $F_2$ as the surface carried by $R$ with invariant measure $W_F=1$, $W_A=1$, $W_{A'}=0$, $W_{A_D}=2$ and $W_{P'_D}=2$. Observe that $F_2$ is a meridional surface in the exterior of $K$ with $2(n+2)$ boundary components: $2n-2$ from $F$, 2 from $A_D$ and 4 from $P'_D$. Denote by $F_j$, $j\geq 3$, the surface carried with positive weights by $R$ with invariant measure $W_F=1$, $W_A=j-1$, $W_{A'}=j-2$, $W_{A_D}=2$ and $W_{P'_D}=j$. Note that 
$F_j$, $j\geq 3$, is a meridional surface in the exterior of $K$ with $2(n+j)$ boundary components: $2n-2$ from $F$, 2 from $A_D$ and $2j$ from $P'_D$.

The surfaces $F_0$ and $F_1$ have genus $g$ and the number of boundary components is $2n$ and $2(n+1)$, respectively. The Euler characteristic of $F_j$, $j\geq 2$, is $W_F\cdot (2-2g -2n)+W_A\cdot 0+W_{A'}\cdot 0 + W_{A_D}\cdot 0 + W_{P'_D}\cdot (-2j)$. That is, $2-2g-2(n+j)$. Hence, as $F_j$ has $2(n+j)$ boundary components, its genus is $g$. We finish the proof of Lemma \ref{lemma: boundaries} with the following proposition.

\begin{prop}\label{proposition:Sg}
The surfaces $F_j$, $j\geq 2$, are essential in $E(K)$.
\end{prop}
\begin{proof}
From Proposition \ref{proposition:incompressible} we have that $R$ is an incompressible branched surface without Reeb components. We also have that the surfaces $F_j$, $j\geq 3$, are carried with positive weights by $R$. Hence, we are under the conditions of Theorem \ref{Oertel}. It then follows that all surfaces carried by $R$ are essential. Therefore, as $F_j$, $j\geq 2$, is carried by $R$, it is essential in $E(K)$.
\end{proof} \end{proof}


\begin{proof}[Proof of Theorem \ref{main}.]
Let us consider a knot $K$ in $\mathcal{K}$. Applying Lemma \ref{lemma: boundaries} with $F$ the 4-punctured sphere $S -N(s\cup t)$ we have the statement of Theorem \ref{main}(a).\\ 
Let $F$ be  a surface of any positive genus and two boundary components as in Lemma \ref{lemma:surfaces}. As $F$ is separating in $B$ and has exactly two boundary components, the boundary components have to be on the regular neighborhood of the same string, which we assume to be $s$. Hence, applying Lemma \ref{lemma: boundaries} with $F$, we have the statement of Theorem \ref{main}(b).  From Lemma \ref{lemma:infinite}, we have that $\mathcal{K}$ is an infinite collection of hyperbolic knots. 
\end{proof}


\section{Appendix}\label{appendix}

In this appendix we observe the existence of n-string atoroidal essential tangles $\mathcal{T}_n$.\\
For a string $s$ in a ball $B$ we can consider the knot obtained by capping off $s$ along $\partial B$, that is by gluing to $s$ an arc in $\partial B$ along the respective boundaries. We denote this knot by $K(s)$. The string $s$ is said to be \textit{knotted} if the knot $K(s)$ is not trivial.\\
Let $s$ be an arc in a ball $B$ such that $K(s)$ is a trefoil, and consider also an unknotting tunnel $t$ for $K(s)$, as in Figure \ref{FTangles1}.

\begin{figure}
\labellist
\small \hair 2pt
\pinlabel (a) at 10 0
\scriptsize
\pinlabel $s$ at 62 18
\pinlabel $B$ at 85 30

\small
\pinlabel (b) at 180 0
\scriptsize
\pinlabel $s$ at 232 18
\pinlabel $t$ at 205 48
\pinlabel $B$ at 255 30

\endlabellist

\centering
\includegraphics[width=0.6\textwidth]{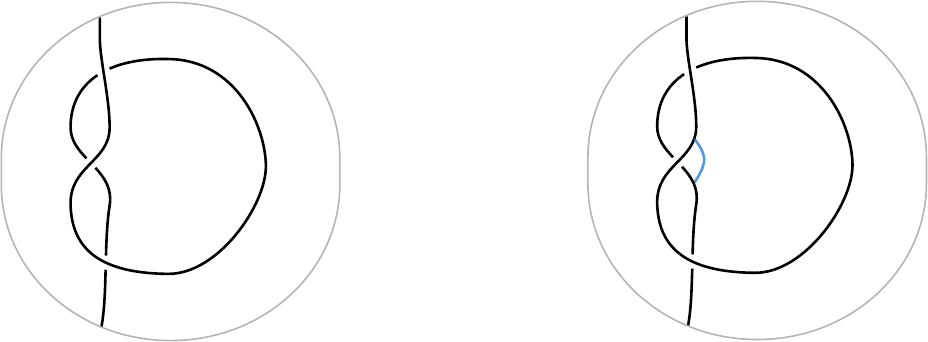}
\caption{: The string $s$ when capped off along $\partial B$ is a trefoil knot with an unknotting tunnel $t$.}
\label{FTangles1}
\end{figure}
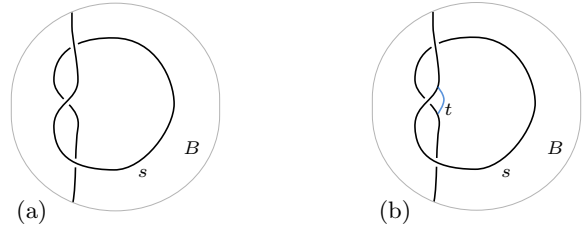

\noindent If we slide $\partial t$ along $s$ into $\partial B$, as illustrated in Figure \ref{FTangles2}(a), we get a new string that we denote by $t_1$, as in Figure \ref{FTangles2}(b).

\begin{figure}[htbp]
\labellist
\small \hair 2pt
\pinlabel (a) at 10 0
\scriptsize
\pinlabel $s$ at 62 18
\pinlabel $B$ at 85 30

\small
\pinlabel (b) at 170 0
\scriptsize
\pinlabel $s$ at 215 17
\pinlabel $t_1$ at 213 48
\pinlabel $B$ at 245 30

\small
\pinlabel (c) at 320 0
\scriptsize
\pinlabel $t_1$ at 350 48
\pinlabel $...$ at 360 48
\pinlabel $t_{n-1}$ at 380 48
\pinlabel $s$ at 375 22
\pinlabel $B$ at 395 30

\endlabellist

\centering
\includegraphics[width=0.9\textwidth]{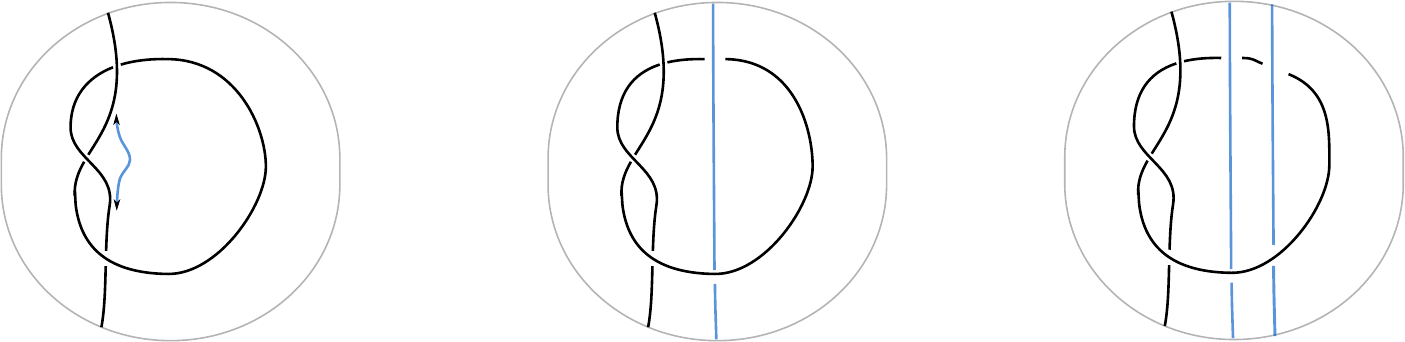}
\caption{: Construction of a $n$-string essential free tangle from $s$ and the unknotting tunnel $t$.}
\label{FTangles2}
\end{figure}
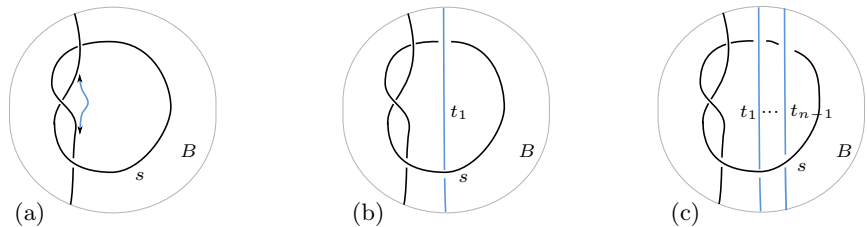

\noindent  The tangle $(B, s\cup t_1)$, denoted $\mathcal{T}_2$, is free by construction. In fact, as $t$ is an unknotting tunnel of $K(s)$, the complement of $N(s)\cup N(t)$ in $B$ is a handlebody. Henceforth, by an ambient isotopy, the complement of $N(s)\cup N(t)$ is also a handlebody. The tangle $\mathcal{T}_3=(B, s\cup t_1\cup t_2)$ is obtained from $\mathcal{T}_2$ by adding a string $t_2$ parallel to $t_1$ in $B-s$. This operation corresponds to a stabilization of the Heegaard splitting of $S^3$ defined by $\partial N(B - s\cup t_1)$. Hence, the exterior of the strings of $\mathcal{T}_3$ in $B$ is a handlebody. That is $\mathcal{T}_3$ is a free tangle. Similarly we define a $n$-string  free tangle $\mathcal{T}_n$ by considering $n-2$ properly embedded arcs $t_2, \ldots, t_{n-1}$ parallel to $t_1$ in $B-s$.

The tangle $\mathcal{T}_2$ is essential. In fact, as $\mathcal{T}$ is free, if there is a disk in $B$ separating $s$ and $t_1$, then both arcs $s$ and $t_1$ are unknotted, which contradicts $s$ being a knotted arc.

The tangle $\mathcal{T}_n$, for $n\in \mathbb{N}_{\geq 3}$ is also essential. In fact, the tangle $(B, s\cup t_i)$ is homeomorphic to $\mathcal{T}$ and henceforth it is essential. Then, if there is a disk $D$ properly embedded in $B$ separating the strings $s\cup t_1\cup t_2$, particular it separates $s$ from $t_i$ for some $i\in \{1, 2\}$, contradicting $(B, s\cup t_i)$ being essential.  


\section*{Acknowledgement}

The author would like to thank John Luecke for discussions on this paper, and Tao Li for some clarifications on branched surface theory.

\end{document}